\documentclass[10pt,english]{amsart}
\def\margin_comment#1{\marginpar{\sffamily{\tiny #1\par}\normalfont}}
\usepackage{graphicx}
\usepackage{setspace}
\usepackage{color}
\usepackage{amsmath,amssymb,amsthm}
\usepackage{epsfig}
\usepackage{tikz}
 \usepackage{amsmath}
\usepackage{tikz-cd}
\usetikzlibrary{matrix, calc, arrows}
\usepackage{tikz}
\usetikzlibrary{arrows,chains,matrix,positioning,scopes}
\usepackage{lipsum}

\makeatletter
\tikzset{join/.code=\tikzset{after node path={%
\ifx\tikzchainprevious\pgfutil@empty\else(\tikzchainprevious)%
edge[every join]#1(\tikzchaincurrent)\fi}}}
\makeatother
\tikzset{>=stealth',every on chain/.append style={join},
         every join/.style={->}}
\tikzstyle{labeled}=[execute at begin node=$\scriptstyle,
   execute at end node=$]
\newtheorem{thm}{Theorem}[section]
\numberwithin{equation}{section} 
\numberwithin{figure}{section} 
\theoremstyle{plain}
\newtheorem*{thm*}{Theorem}
\theoremstyle{definition}
\theoremstyle{plain}
\newtheorem{thm_A}{Theorem}
\newtheorem*{defn*}{Definition}

\theoremstyle{plain}

\theoremstyle{plain} 

\theoremstyle{plain}
\newtheorem{prop}[thm]{Proposition} 
\theoremstyle{remark}

\theoremstyle{remark}
\newtheorem{rem}[thm]{Remark}
\theoremstyle{plain}

\theoremstyle{plain}

\theoremstyle{plain}
\newtheorem{lem}[thm]{Lemma}
\newtheorem*{lem*}{Lemma} 
\theoremstyle{definition}
\newtheorem{defn}[thm]{Definition}

\newtheorem*{acknowledgment*}{Addentum}

\theoremstyle{plain}
\newtheorem*{ex*}{Example}
\theoremstyle{plain}

\AtBeginDocument{
  
}

\begin{document}
\title{Left orders in Garside groups}
\author{Fabienne Chouraqui}
\maketitle
\begin{abstract}
We consider the structure group of a non-degenerate  symmetric (non-trivial) set-theoretical solution of the quantum Yang-Baxter equation. This is a Bieberbach group and also a Garside group. We show  this group is not bi-orderable, that is it does not admit a total order which is invariant under left and right multiplication. Regarding the existence of a left invariant total ordering, there is a great diversity. There exist   structure groups with a recurrent left order and with  space of left orders homeomorphic to the  Cantor set, while there exist others that are even not unique product groups.
\end{abstract}
\section*{Introduction}
A group $G$ is \emph{left-orderable} if there exists a strict  total ordering  $\prec$ of its elements which  is invariant under left multiplication, that is  $g \prec h$  implies $fg \prec fh$ for all $f,g,h$ in $G$. If the order $\prec$ is also invariant under right multiplication, then $G$ is said to be \emph{bi-orderable}. The braid group $B_n$,  with $n \geq 3$ strands, is left-orderable but  not bi-orderable \cite{deh_braids_order}, and if $n \geq 5$ none of these orders is Conradian \cite{rhem-rolfsen}. In \cite{deh-book-order}, the  question whether every Garside group is left-orderable is  araised (Question 3.3, p.292, also in \cite{deroin}). It is a very natural question as the Garside groups extend the braid groups in many respects and  it motivated  our research in the  context of  the structure group of a non-degenerate  symmetric set-theoretical solution of the quantum Yang-Baxter equation. This group is  a Garside group that satisfies many interesting properties \cite{chou_art}, \cite{ddkgm}. In this note, we show this group is not bi-orderable and we find the question whether it is left-orderable has  a wide range of answers. We state our main results and  refer to Sections 1,2 for definitions:

\begin{thm_A}\label{thm_intro_1}
Let $G(X,S)$ be the structure group of a non-degenerate  symmetric (non-trivial) set-theoretical solution $(X,S)$ of the quantum Yang-Baxter equation. Then $G(X,S)$ is not bi-orderable. Furthermore, $G(X,S)$ has generalized torsion elements.
\end{thm_A}

\begin{thm_A}\label{thm_intro_2}
Let $G(X,S)$ be the structure group of a non-degenerate,  symmetric (non-trivial) set-theoretical solution $(X,S)$ of the quantum Yang-Baxter equation. Assume $(X,S)$ is a retractable solution and $\mid X \mid \geq 3$. Then \\
$(i)$ $G(X,S)$ has a recurrent left order.\\
$(ii)$ The space of left orders of $G(X,S)$ is  homeomorphic to the  Cantor set.\\
$(iii)$  $G(X,S)$ has an infinite number  of Conradian  left orders. 
\end{thm_A}
Note that under the assumptions of Theorem \ref{thm_intro_2}, $G(X,S)$ is locally indicable (each non-trivial finitely generated subgroup has a quotient isomorphic to $\mathbb{Z}$), as the existence of a  recurrent left order implies local indicability \cite{morris}.  In contrast, for $n \geq 5$, no braid group $B_n$ is  locally indicable \cite{deh-book-order}[p.287] and hence $B_n$ has no recurrent left order like most of the left-orderable groups.   E. Jespers and J. Okninski prove the structure group of a retractable solution is poly-(infinite)cyclic \cite{jespers_book}[p.223] and poly-(infinite)cyclic   implies locally indicable. \\
Here is an outline of the paper. Section 1 provides some standard definitions on orderable groups. Section 2 introduces the structure group of a non-degenerate,  symmetric set-theoretical solution $(X,S)$ of the quantum Yang-Baxter equation. Section 3 proves Theorem \ref{thm_intro_1}. Section 4 proves Theorem \ref{thm_intro_2} and concludes the paper with some remarks and questions. 
\section{Preliminaries on groups ordering}
We introduce some definitions and refer to   \cite{book-order}, \cite{glass}, \cite{deh-book-order}, \cite{remtulla}, \cite{deroin}, \cite{levi}, \cite{linnel}  and survey \cite{rolfsen}. A group $G$ is \emph{left-orderable} if there exists a  strict total ordering  $\prec$ of its elements which  is invariant under left multiplication, that is  $g \prec h$  implies $fg \prec fh$ for all $f,g,h$ in $G$.  If a group $G$ is left-orderable, then it satisfies the  unique product property, that is for any finite subsets $A,B \subseteq G$, there exists at least one element $x \in AB$ that can be uniquely written as $x=ab$, with $a \in A$ and  $b \in B$. We call  a strict total ordering   which  is invariant under left multiplication a  \emph{left order}. The  \emph{positive cone of} a left order $\prec$ is defined by $P=\{g \in G \mid 1 \prec g\}$ and it satisfies:\\
$(1)$ $P$ is a semigroup, that is $P\cdot P \subseteq P$\\
$(2)$ $G$ is partitioned by $P$, that is $G= P \cup P^{-1} \cup \{1\}$ and  $P \cap P^{-1}=\emptyset$\\
 Conversely, if there exists a subset $P$ of $G$ that satisfies $(1)$ and $(2)$, then $P$ determines a unique  left order $\prec$ defined by $g \prec h$ if and only if $g^{-1}h \in P$. A subgroup $N$ of a left-orderable group $G$ is called \emph{convex} (with respect to $\prec$), if for any $x,y,z \in G$ such that $x,z \in N$ and $x \prec y \prec z$, we have $y \in N$. A left order $\prec$ is \emph{Conradian} if for any strictly positive elements  $a,b \in G$, there is  $n\in \mathbb{N}$ such that $b \prec ab^n$. A left-orderable group $G$ is called \emph{Conradian} if it admits a Conradian left order. Conradian left-orderable groups share many of the properties of the bi-orderable groups. A left order $\prec$ in a countable group $G$ is \emph{recurrent (for every cyclic subgroup)} if for every  $g \in G$ and every finite increasing sequence $h_1 \prec h_2\prec ... \prec h_r$ with  $h_i \in G$, there exists $n_i \rightarrow \infty$ such that $\forall i, \;h_1 g^{n_i}\prec h_2 g^{n_i}\prec ... \prec h_r g^{n_i}$ (see \cite{morris}[Defn.3.2], \cite{linnel_morris}[Defn.3.1]). A recurrent left order is Conradian \cite{morris}. 
 
 The set of all left orders of a group $G$ is denoted by $LO(G)$ and it is a topological space (compact and  totally disconnected with respect to the topology induced by the Tychonoff topology on the power set of $G$) \cite{sikora}. If $G$ is left-orderable, it acts on  $LO(G)$ by conjugation: the image of  $\prec$ under $g \in G$ is  $\prec_g \in LO(G)$ defined by $a \prec_g b$ if and only if $gag^{-1} \prec gbg^{-1}$.  
  A left order $\prec$ is \emph{finitely determined} if there is a finite subset $\{g_1,g_2,...,g_k\}$ of $G$ such that $\prec$ is the unique left-invariant ordering of $G$ satisfying $1 \prec g_i$ for $1 \leq i\leq k$. A finitely determined left order $\prec$ is also called \emph{isolated}, since  $\prec$ is finitely determined if and only if it is not a limit point of $LO(G)$. If the positive cone of $\prec$ is a finitely generated semigroup, then $\prec$ is isolated.
   The set  $LO(G)$ cannot be countably infinite \cite{linnel}.  If $G$ is a countable left-orderable group,   $LO(G)$ is either finite, or homeomorphic to the Cantor set, or homeomorphic  to a subspace of the Cantor space with isolated points. Furthermore, $LO(G)$ is homeomorphic to the Cantor set if and only if it is nonempty and no left-invariant ordering of $G$ is isolated \cite{deh-book-order}[p.267].  If $G$ is a countable and virtually solvable left-orderable group,  then  $LO(G)$  is either finite or a Cantor set \cite{rivas}.
An element $x$ of a subset $F$ of $G$ is an \emph{extreme point of $F$} if, for all $g \in
 G\setminus \{1\}$, either $ga$ or $g^{-1}a$ is not in $F$. A group  $G$  is \emph{diffuse} if every non-empty finite subset $F$ of $G$  has an  extreme point \cite{bowditch}, \cite{linnel_morris}.  A diffuse group satisfies the unique product property  \cite{bowditch}. This notion is strictly weaker than left-orderability  \cite{rimbaut} and it is equivalent to the notion of locally invariant left order \cite{linnel_morris}. We refer to  \cite{linnel_morris}[Section 6].
\section{Set-theoretical solutions of the Yang-Baxter equation}\label{sec_QYBE}
\label{subsec_qybe_Backgd}
Fix a finite dimensional vector space~$V$ over the field~$\mathbb{R}$. The quantum Yang-Baxter equation (QYBE) on~$V$ is the equality $R^{12}R^{13}R^{23}=R^{23}R^{13}R^{12}$ of linear transformations on $V  \otimes V \otimes V$, where ~$R: V  \otimes V\to V  \otimes V$ is a linear operator and $R^{ij}$ means $R$ acting on the $i$th and $j$th components. A \emph{set-theoretical solution} of this equation is a pair~$(X,S)$ such that $X$ is a basis for $V$ and $S : X \times X \rightarrow X \times X$ is a bijective map.
The map $S$ is defined by  $S(x,y)=(g_{x}(y),f_{y}(x))$, where $f_x, g_x:X\to X$ are functions  for all  $x,y \in X$. The pair~$(X,S)$ is \emph{non-degenerate} if for any  $x\in X$, $f_{x}$ and $g_{x}$ are bijective. It is  \emph{involutive} if $S\circ S = Id_X$, and  \emph{braided} if $S^{12}S^{23}S^{12}=S^{23}S^{12}S^{23}$, where the map $S^{ii+1}$ means $S$ acting on the $i$-th and $(i+1)$-th components of $X^3$.
It is said to be \emph{symmetric} if  it is involutive and braided.
 Let $\alpha:X \times X \rightarrow X\times X$ be defined by $\alpha(x,y)=(y,x)$, and let $R=\alpha \circ S$, then $R$ satisfies the QYBE if and only if  $(X,S)$ is  braided. We follow~\cite{etingof}, \cite{jespers_book} and refer to \cite{gateva_van,jespers} for more details, and  to \cite{chou_art,chou_godel2,chou_coxeter} for examples.
\begin{defn} \label{def_struct_gp} Let $(X,S)$ be a non-degenerate symmetric set-theoretical solution. The \emph{structure group} of $(X,S)$ is $G(X,S)=\operatorname{Gp} \langle X\mid\ xy = g_x(y)f_y(x)\ ;\ x,y\in X \rangle\label{equation:structuregroup1}$. 
\end{defn}
 \begin{defn}\label{def_decomposale}\cite{etingof}
   Let $(X,S)$ be  a non-degenerate symmetric solution.\\
  $(i)$ A subset $Y$ of $X$ is an \emph{invariant} subset if $S(Y \times Y)\subseteq Y \times Y$.\\
  $(ii)$ An invariant subset $Y$  is \emph{non-degenerate} if $(Y,S\mid_{Y\times Y})$ is  non-degenerate and symmetric. \\
  $(iii)$ The solution $(X,S)$ is  \emph{decomposable} if $X$ is the union of two non-empty disjoint non-degenerate invariant subsets. Otherwise, $(X,S)$ is \emph{indecomposable}.
  \end{defn} 
    A solution $(X,S)$ is \emph{trivial} if $g_{x}=f_{x}=Id_X$, $\forall x \in X$, and  its structure group is $\mathbb{Z}^{\mid X \mid}$. Non-degenerate symmetric solutions (up to equivalence) are in one-to-one correspondence with quadruples $(G,X,\rho, \pi)$, where $G$ is a group, $X$ is a set, $\rho$ is a left action of $G$ on $X$, and $\pi$ is a bijective $1$-cocycle of $G$ with coefficients in $\mathbb{Z}^{X}$, where $\mathbb{Z}^{X}$ is the free abelian group spanned by $X$ \cite{etingof}. Indeed, $G(X,S)$ is naturally a subgroup of $\operatorname{Sym(X)} \ltimes \mathbb{Z}^{X}$, such that the $1$-cocycle defined by the projection $G(X,S) \rightarrow \mathbb{Z}^{X}$ is bijective. The product in $\operatorname{Sym(X)} \ltimes \mathbb{Z}^{X}$ is defined by: $f_x^{-1}t_x \, f_y^{-1}t_y= f_x^{-1} f_y^{-1}t_{f_y(x)}t_y$.
      More precisely:
  \begin{thm}\cite{etingof}\label{prop_etingof}
   Let $(X,S)$ be  a non-degenerate symmetric  solution and $G(X,S)$ be its structure group. Let $\operatorname{Sym(X)} $ be the group of permutations of $X$ and $\mathbb{Z}^{X}$ be the free abelian group spanned by $X$. Let the map $\phi_f: G(X,S) \rightarrow  \operatorname{Sym(X)} \ltimes \mathbb{Z}^{X}$ be defined by $\phi_f(x)= f_{x}^{-1}\,t_x$, where $x \in X$ and $t_x$ is  the generator of $\mathbb{Z}^{X}$ corresponding to $x$. Then\\
   $(i)$ The assignment $x \mapsto f_x^{-1}$ is a left action of $G(X,S)$ on $X$.\\
    $(ii)$ Let $a \in G(X,S)$ and $w=m_1t_1+m_2t_2+...+m_nt_n \in \mathbb{Z}^{X}$. Assume  $a$ acts on $X$ via the permutation $f$. Then $a$ acts on $\mathbb{Z}^{X}$ in the following way: $a \bullet t_x= t_{f(x)}$ and  $a \bullet w= m_1t_{f(1)}+m_2t_{f(2)}+...+m_nt_{f(n)}$,  where $\bullet$ denotes the extension of the left action  of $G(X,S)$ on $X$ defined in $(i)$ to $\mathbb{Z}^{X}$.\\
   $(iii)$ The map  $\phi_f$  is a monomorphism. \\
   $(iv)$ The map $\pi: G(X,S) \rightarrow \mathbb{Z}^{X}$, defined by $\pi(g)= w$ if $\phi_f(g)= \alpha\, w$, with $\alpha \in \operatorname{Sym(X)}$ and  $w \in \mathbb{Z}^{X}$,  is a bijective $1$-cocycle satisfying $\pi(a_1a_2)=(a_2^{-1}\bullet \pi(a_1))\,+\pi(a_2)$.
  \end{thm}
  A \emph{crystallographic group} is a discrete cocompact subgroup of  the group of  isometries of $\mathbb{R}^n$.  A \emph{Bieberbach group} is a torsion-free crystallographic group. The structure group $G(X,S)$ of a non-degenerate symmetric solution $(X,S)$ with $\mid X \mid =n$ is a Bieberbach group of rank  $n$ \cite{gateva_van}. Indeed,  $G(X,S)$ acts freely on $\mathbb{R}^n$ by isometries with fundamental domain $[0,1]^n$ (see  \cite{gateva_van}, \cite{jespers_book}[p.218]). The structure groups satisfy another property,   that makes this family of groups particularly interesting. Indeed, every structure group is a \emph{Garside group} \cite{chou_art}, \cite{ddkgm}, that is a group of fractions of a cancellative monoid $M$ which is a lattice with respect to  left-divisibility and with a Garside element $\Delta$ (the left and right generators of $\Delta$ coincide, are finite in number and generate $M$). Garside groups were defined as a generalisation in some extend of the braid groups and the finite-type Artin-Tits groups \cite{DePa},  \cite{ddkgm}.
       
  \section{The structure group is not bi-orderable}\label{sec_not_biorder}
   If $(X, S)$ is  a non-degenerate symmetric set-theoretical solution, with $\mid X \mid =n$, that satisfies a certain  condition $C$, then there is a short exact sequence $1 \rightarrow N \rightarrow G(X,S)  \rightarrow W  \rightarrow 1$, where $N$ is a normal free abelian subgroup of rank $n$ and $W$ is a finite group of order $2^n$. Moreover, $W$  is  a Coxeter-like group, that is $W$ is  a finite quotient that plays the role the pure braid group $P_n$ plays in the sequence $1 \rightarrow P_n \rightarrow B_n  \rightarrow S_n  \rightarrow 1$, where $B_n$ is the braid group and $S_n$  the symmetric group or more generally the role Coxeter groups play for the finite-type Artin groups   \cite{chou_godel2}. In \cite{deh_coxeterlike}, it is proved that the condition $C$ may be relaxed and  that for each  $(X,S)$ there is a natural number $m$ such that for each $x \in X$ there is a chain of trivial relations of the form $xy_1=xy_1,\;y_1y_2=y_1y_2,\;y_2y_3=y_2y_3,...,y_{m-1}x=y_{m-1}x$, $y_i \in X$.   Here, we  call the element $y_1y_2y_3...y_{m-1}x$ the \emph{frozen element of length $m$ ending with $x$} and denote it by $\theta_x$. The subgroup $N$, generated by the $n$ frozen elements of length $m$,   $y_1y_2y_3...y_{m-1}x,\;xy_1y_2y_3...y_{m-1},\;..., y_2y_3...y_{m-1}xy_1$, is  normal, free abelian  of rank $n$ and the group $W$ defined by $G(X,S)/N$ is finite of order $m^n$ and is a  Coxeter-like  group \cite{deh_coxeterlike}.
   \begin{lem}\label{lem_conjug_frozen}
     Let $(X,S)$ be  a non-degenerate symmetric set-theoretical solution, with $\mid X \mid =n$ and with permutations $f_1,...,f_n,g_1,...,g_n$. Let   $G(X,S)$ be its structure group. Let $\theta_1,...,\theta_n$ be the $n$ frozen elements in  $G(X,S)$ and  $x_k$ be an element of $X$. Then\\  $(i)$ $x_k\theta_ix_k^{-1}=\theta_{f^{-1}_k(i)}$.\\
     $(ii)$ If $(X,S)$ is not the trivial solution, then there exists $1 \leq i \leq n$, such that $C_{G(X,S)}(\theta_i)$,   the centralizer of $\theta_i$, is not $G(X,S)$.
    \end{lem}
      \begin{proof}
      $(i)$ We show $\pi(x_k\theta_i)=\pi(\theta_{f^{-1}_k(i)}x_k)$, where $\pi: G(X,S) \rightarrow \mathbb{Z}^{X}$ is the bijective $1-$cocycle defined in Theorem \ref{prop_etingof}. Indeed, $\pi(x_k\theta_i)= \theta_i^{-1}\bullet\pi(x_k)+\pi(\theta_i)=\pi(x_k)+\pi(\theta_i)=t_k +mt_i$, using first Theorem \ref{prop_etingof}$(iv)$, next the fact that the frozen elements act trivially on $\mathbb{Z}^n$ and that $\pi(\theta_i)=mt_i$ (see \cite{deh_coxeterlike} lemma 3.2 or \cite{chou_coxeter} lemma 1.12 written multiplicatively there). In the same way,  $\pi(\theta_{f^{-1}_k(i)}x_k)= x_k^{-1}\bullet \pi(\theta_{f^{-1}_k(i)}) +\pi(x_k)= x_k^{-1}\bullet (mt_{f^{-1}_k(i)}) +t_k=mt_{f_k(f^{-1}_k(i))} +t_k=mt_i+t_k$.\\
       $(ii)$ From $(i)$,  $C_{G(X,S)}(\theta_i)=G(X,S)$ for all  $1 \leq i \leq n$, if and only if $f^{-1}_k(i)=i$ for all  $1 \leq i,k \leq n$, that is if and only if $(X,S)$ is  the trivial solution.
      \end{proof} 
      An element $g \neq 1$ of a group $G$ is called a \emph{generalized torsion element} if there exist $h_i \in G$ such that $\prod\limits_{i=1}^{n} h_igh_i^{-1}=1$. A group with generalized torsion element cannot be bi-orderable.
  \begin{proof} \emph{of Theorem $1$}
   Assume by contradiction there exists   a total ordering $<$ of $G(X,S)$  invariant under left and right multiplication. Let $\theta_1,...,\theta_n$ denote the $n$ frozen elements in  $G(X,S)$. From lemma \ref{lem_conjug_frozen}, there exits  $x_k \in X$ and a frozen element $\theta_i$ such that $x_k$ and $\theta_i$ do not commute and  furthermore, $x_k\theta_ix_k^{-1}=\theta_{f^{-1}_k(i)}$. 
   With no loss of generality, suppose $\theta_i\,<\,\theta_{f^{-1}_k(i)}$. So, from our assumption,   $x_k\theta_ix_k^{-1}\,<\,x_k\,\theta_{f^{-1}_k(i)}\,x_k^{-1}$, that is $\theta_{f^{-1}_k(i)}\, <\,\theta_{f^{-2}_k(i)}$ from lemma \ref{lem_conjug_frozen}$(i)$. Inductively,  $\theta_i\,<\,\theta_{f^{-1}_k(i)}\, <\,\theta_{f^{-2}_k(i)}\,<...$. But $f_k$ has finite order, so there exists a natural number $p$ such that $\theta_{f^{-p}_k(i)}=\theta_i$ and this  contradicts our assumption.
   The commutator $[x_k,\theta_i]$ is a generalized torsion element, since on one hand $[x_k,\theta_i]\neq 1$ and $[x_k^p,\theta_i]=1$ and on second hand $[x_k^p,\theta_i]$
  is the product of conjugates of $[x_k,\theta_i]$, using $[x^m,y]=x[x^{m-1},y]x^{-1}[x,y]$. 
   \end{proof}
\section{A family of structure groups with a recurrent  left order}
 Let $G(X,S)$ be the structure group of a non-degenerate,  symmetric  set-theoretical solution $(X,S)$,  $\mid X \mid =n$.  Let $\epsilon: G(X,S)\rightarrow\mathbb{Z}$ denote the augmentation map,  defined by $\epsilon(x_{i_1}^{m_1}x_{i_2}^{m_2}...x_{i_k}^{m_k})=\sum_{i=1}^{i=k} m_i$. The map $\epsilon$ is an epimorphism (from the form of the defining relations), so $G(X,S)$  is indicable and there is an exact sequence $1 \rightarrow K\rightarrow G(X,S)\rightarrow \mathbb{Z}\rightarrow 1$, with $K=ker(\epsilon)$. 
 Assume  $(X,S)$ satisfies  additionally: all the permutations $f_1,...,f_n$ are equal.  We show that for such solutions the restriction of  the bijective 1-cocycle $\pi:G(X,S)\rightarrow \mathbb{Z}^n$ to the normal subgroup $K$  is an isomorphism of groups. Note that for each $n$, such solutions exist and if all the $f_i$, $1 \leq i \leq n$, are equal to a cycle of length $n$, the solution is called a \emph{permutation solution} \cite{etingof}.
 \begin{prop}\label{prop_permut_sol}
 Let $G(X,S)$ be the structure group of a non-degenerate,  symmetric  (non trivial) set-theoretical solution $(X,S)$, $\mid X \mid =n$, such that all the permutations $f_1,...,f_n$ are equal. Let $K$ denote the kernel of  $\epsilon: G(X,S)\rightarrow\mathbb{Z}$. Let $LO(G(X,S))$ denote the space of left orders of $G(X,S)$. Assume $n\geq 3$. Then\\
 $(i)$ The restriction of $\pi$ to  $K$  is an isomorphism of groups and $K \simeq \mathbb{Z}^{n-1}$. \\
 $(ii)$ The normal subgroup $K$ is convex with respect to an infinite number of left orders.\\
 $(iii)$ There exists a recurrent left order in $LO(G(X,S))$.\\
 $(iv)$ $LO(G(X,S))$ is homeomorphic to the Cantor set with  no isolated element. \\
  $(v)$ Every  left order of  $G(X,S)$ is  Conradian.
 \end{prop}
 \begin{proof}
  $(i)$ Let $a_1,a_2 \in K$. From Thm. \ref{prop_etingof},  $\pi(a_1a_2)=(a_2^{-1}\bullet \pi(a_1))\,+\pi(a_2)$. As $\epsilon(a_2^{-1})=\epsilon(a_2)=0$ and $f_1,...,f_n$ are equal to some $f$, $a_2$ acts trivially  on $\mathbb{Z}^n$, so $\pi(a_1a_2)= \pi(a_1)\,+\pi(a_2)$. As, $\pi$ is also bijective, it is  an isomorphism of groups and clearly $K \simeq \mathbb{Z}^{n-1}$. \\
  $(ii)$ Since $1 \rightarrow K\rightarrow G(X,S)\rightarrow^{\epsilon} \mathbb{Z}\rightarrow 1$ and $K$, $\mathbb{Z}$  are (left-)orderable,  $G(X,S)$ is left-orderable \cite{book-order}[p.26]. Furthermore, each  (left-)order of $K$ induces a left order of $G(X,S)$ in the following way:  given $g,h \in G(X,S)$,  $g \prec h$ if $\epsilon(g) \prec_{\mathbb{Z}}\epsilon(h)$ and if $\epsilon(g) =\epsilon(h)$, $0 \prec_{K} g^{-1}h$.   As there is an uncountable number of orders in $K$ \cite{book-order}[p.43] and each order of $K$ induces  a left order of $G(X,S)$, there is an infinite number of  left orders of $G(X,S)$ and from the definition of $\prec$, $K$ is  convex (indeed, if $x \prec y \prec z$ and $x,z \in K$, then $\epsilon(y)=0$).\\
  $(iii)$ $G(X,S)$  is solvable \cite{etingof} and  left-orderable, so it has a recurrent left order \cite{morris}. \\
  $(iv)$ The space of left orders of  a countable
  (virtually) solvable group is either finite or homeomorphic to the Cantor
  set \cite{rivas}, so  $LO(G(X,S))$ is homeomorphic to  the Cantor set.   Furthermore, for a countable group $G$, the space $LO(G)$ is homeomorphic to the Cantor set if and only if it is nonempty and no left order is isolated \cite{deh-book-order}[p.267], so no  left  order in $LO(G(X,S))$ is isolated. Note, for $n \geq 2$, $\mathbb{Z}^n$ admits no isolated orders \cite{sikora}.\\
  $(v)$ Every left order of $G(X,S)$  is also right invariant on a  subgroup of finite index (the free abelian  subgroup of finite index $N$ from Section \ref{sec_not_biorder}), so it is Conradian \cite{deh-book-order}[p.288]. 
  \end{proof}
 \begin{rem}
In case $n=2$, there are exactly two solutions: the trivial one with structure group $\mathbb{Z}^2$ and the permutation solution with structure group the Klein-bottle group (presented by $\langle a,b\mid a^2=b^2 \rangle$). The construction of left orders from Proposition  \ref{prop_permut_sol} works also in this case, but as $K=\mathbb{Z}$, there are only four left orders  induced, which are Conradian. These are all the only left orders and they are  isolated (see \cite{deroin}[p.54] for details). The question arises whether,  for $n \geq 3$, the  construction of left orders from Proposition \ref{prop_permut_sol} provides also all the left orders.
 \end{rem} 
 Let  $(X,S)$ be a  non-degenerate,  symmetric  solution with structure group $G(X,S)$. A \emph{retract relation} $\equiv$ on $X$ is a  congruence relation defined by $x_i \equiv x_j$ if and only if  $f_i=f_j$. The quotient group $G(X,S)/\equiv$ is denoted by $\operatorname{Ret}^1(G)$ and it is also  the structure group of a non-degenerate,  symmetric  solution with set  $X/\equiv$ and function $S/\equiv$ induced accordingly. The kernel of the canonical homomorphism $G(X,S) \rightarrow G(X,S)/\equiv $ is a finitely generated torsion-free abelian group \cite{jespers_book}[p.222]. For any integer $m \geq 1$,  $\operatorname{Ret}^{m+1}(G)=\,\operatorname{Ret}^1(\operatorname{Ret}^m(G))$. The  solution  $(X,S)$  is  called \emph{retractable} if there exits  $m \geq 1$ such that $\operatorname{Ret}^m(G)$ is a cyclic group and if $m$ is the smallest such integer,  $(X,S)$  is  called \emph{retractable (or multipermutation solution) of level} $m$. A  solution $(X,S)$, $\mid X \mid =n$, for which  all the permutations $f_1,...,f_n$ are equal is retractable  of level $1$. We refer to \cite{etingof}, \cite{jespers_book} for details. 
  \begin{proof} \emph{of Theorem $2$}
  $(i),(ii),(iii)$
  Assume $(X,S)$ is retractable of level $m$. The proof is by induction on $m$. The case  $m=1$ is proved in Prop. \ref{prop_permut_sol}. For $m \geq 2$, we have the exact sequence $1 \rightarrow \operatorname{Ker}(\equiv) \hookrightarrow G(X,S) \rightarrow G(X,S)/ \equiv \,\rightarrow 1$, where   $G(X,S)/ \equiv$ is a retractable solution of level $m-1$.
 Since $\operatorname{Ker}(\equiv)$  is a torsion free abelian subgroup of $G(X,S)$, it is a (left-)orderable group \cite{levi} and from the induction assumption, $G(X,S)/ \equiv$ is left-orderable, so $G(X,S)$ is also left-orderable \cite{book-order}[p.26]. Since each left order in   $G(X,S)$ is induced by a Conradian left order in $G(X,S)/ \equiv$ and all the orders in $\operatorname{Ker}(\equiv)$  are Conradian, $G(X,S)$ has also an infinite  set of Conradian left orders. 
 As $G(X,S)$  is solvable \cite{etingof}, it   has a recurrent left order \cite{morris} and   $LO(G(X,S))$ is homeomorphic to  the Cantor set \cite{rivas}. 
   \end{proof} 
   Some remarks to conclude. There are the following implications: Bi-orderable $\Rightarrow$ Recurrent left-orderable  $\Rightarrow$ Locally indicable $\Rightarrow$ Left-orderable $\Rightarrow$ Diffuse  $\Rightarrow$ Unique product $\Rightarrow$ Torsion-free and  Unique product $\Rightarrow$  Kaplansky's  Unit conjecture satisfied (the units in the group algebra are trivial) $\Rightarrow$  Kaplansky's Zero-divisor  conjecture satisfied (there are no zero divisors in the group algebra) $\Rightarrow$  Kaplansky's  Idempotent conjecture satisfied  (there are no non-trivial idempotents in the group algebra).
 We found there are no (non-trivial) solutions with structure group bi-orderable, and all the retractable solutions admit a  recurrent  left order. For non-retractable  solutions, the kernel of  $\epsilon: G(X,S)\rightarrow\mathbb{Z}$ is not necessarily a free abelian group and the methods from Proposition \ref{prop_permut_sol} cannot be applied, so, the question arises whether there exist structure groups of  non-retractable  solutions that are left-orderable.  In fact, there exist  non-retractable  solutions with   structure group  a non unique product group. Indeed,  E. Jespers and J. Okninski  give an example of structure group of a non-retractable  solution (with $n=4$) which is not a unique product group  \cite{jespers_book}[p.224], as they prove this group contains a subgroup isomorphic to the Promislow group. The Promislow group is  a non unique product Bieberbach group \cite{promislow}, which does not belong to the class of the structure groups \cite{jespers_book}[p.224]. 
  So,  this answers in the negative  the  question whether every Garside group is left-orderable and furthermore provides an example of non unique product Garside group. So, being a Garside group does not imply being locally indicable, nor left-orderable, nor unique product, but a Garside group is necessarily torsion free \cite{deh_torsion}. 
  
  The structure group of a non-degenerate, symmetric solution enjoys  another interesting particularity, it is a Bieberbach group \cite{gateva_van}, \cite{jespers_book}, that is it is a torsion free crystallographic group.  A classification of Bieberbach groups of small dimension (up to four) in relation with the existence of a left order is given in \cite{rimbaut}[p.13]. Bieberbach groups  satisfy Kaplansky's  zero divisor conjecture, as it holds  for all torsion-free finite-by-solvable groups \cite{zerodiv}. As the braid groups are left-orderable, they also  satisfy the zero divisor conjecture, so  Question $3.3$ from \cite{deh-book-order} might be replaced by: does a Garside group satisfy  Kaplansky's  zero divisor conjecture? It is still unknown whether Kaplansky's  unit conjecture holds in Bieberbach groups. D. Craven and P. Pappas study the question whether the unit conjecture holds in the Promislow group  (also called  Passman group) (see \cite{craven} for some preliminary results). 
  
  Amongst the $23$ solutions with $n=4$, there are two  non-retractable (indecomposable) solutions and the structure group of one of them is presented by $\operatorname{Gp}\langle x_1,x_2,x_3,x_4 \mid x_1^2=x_2^2\,,\, x_1x_2=x_3^2\,,\,x_2x_1=x_4^2\,,\,x_1x_3=x_4x_1\,,\,x_2x_4=x_3x_2\,,\,x_3x_4=x_4x_3\rangle$ with $g_1=(1,2,3,4),$ \\$g_2=(1,4,3,2), g_3=(1,3), g_4=(2,4); f_1=(1,2,4,3), f_2=(1,3,4,2), f_3=(2,3), f_4=(1,4)$. The second one is given in  \cite{jespers_book}[p.224]. If a  decomposable solution contains a  non-retractable solution with non unique product structure group, then it is  also non-retractable and it has  non unique product structure group. So, the question is what about the class   of indecomposable  non-retractable  solutions? For $5\leq n \leq 7$, all the  indecomposable  solutions are retractable and for $n=8$, amongst the $34528$ solutions, there are $47$ indecomposable  non-retractable  solutions \cite{etingof}. As the structure groups have a wide range of behaviours, the intriguing question is whether  amongst the indecomposable  non-retractable  solutions, there  are special cases of groups. More specifically, can we find  there groups that are left-orderable, or diffuse? Or, groups that are unique product but not left-orderable? Or,  diffuse but not left-orderable?

\bigskip\bigskip\noindent
{ Fabienne Chouraqui,}\smallskip\noindent
University of Haifa at Oranim, Israel.
\smallskip\noindent\\
E-mail: {\tt fabienne.chouraqui@gmail.com,  fchoura@sci.haifa.ac.il}

\end{document}